\documentclass[12pt]{amsart}
\usepackage{amsmath, amssymb,tabularx,graphicx}
\usepackage{cite}
\usepackage[mathscr]{eucal}
\newtheorem{thm}{Theorem}[section]

\newtheorem{lemma}[thm]{Lemma}
\newtheorem{defn}[thm]{Definition}

\newtheorem{cor}[thm]{Corollary}
\newtheorem{rk}{Remark}

\newtheorem*{thma}{Theorem A}
\newtheorem*{thmb}{Theorem B}
\newtheorem*{ack}{Acknowledgements}
\newcommand{\Z}{\mathscr{Z}}

\begin{document}

\title {Finite-Dimensionality of Z-Boundaries} 
\author{ Molly A. Moran}

\begin{abstract} 

In this paper, we refine the notion of Z-boundaries of groups introduced by Bestvina and further developed by Dranishnikov. We then show that the standard assumption of finite-dimensionality can be omitted as the result follows from the other assumptions. 
\end{abstract}
%

\maketitle

\section{Introduction}

It is easy to construct a proper CAT(0) space with infinite-dimensional boundary, but a result by Swenson \cite{Swe99} shows that such a space cannot admit a cocompact action by isometries. 
This observation mirrors an earlier theorem by Gromov \cite{Gr87} which asserts that boundaries of hyperbolic groups are finite-dimensional. 

The rich study of CAT(0) and hyperbolic group boundaries led Bestvina to formalize the concept of group boundaries for wider classes of groups \cite{Be96}. Included in his definition is a hypothesis which forces these boundaries, known as $\mathscr{Z}$-boundaries, to be finite-dimensional. Later, when Dranishnikov generalized Bestvina's work  to allow for groups with torsion \cite{Dr06}, he omitted the requirement in Bestvina's original definition that forced the boundaries to be finite-dimensional. 

%
In this paper, we prove a generalization of Swenson's theorem that applies to  a more general class of spaces. A consequence of this result is a more unified treatment of group boundaries put forth by Bestvina and Dranishnikov. We show that there is no advantage in restricting our attention to finite-dimensional spaces as in \cite{Be96}. 
In regards to \cite{Dr06}, all conclusions about the cohomological dimension of group boundaries can be extended to results about the Lebesgue covering dimension of these boundaries. 

We close with statements of our main results, which may be found in Sections 3 and 4, respectively. 
\begin{thma} If a proper metric ANR $(X,d)$ admits a metric $\Z$-structure $(\hat{X},Z)$, then $Z$ is finite-dimensional. 
\end{thma}
\begin{thmb} If a torsion-free group $G$ admits an AR $\Z$-structure, then $G$ admits a $\Z$-structure, as defined in \cite{Be96}. 
\end{thmb}

\begin{ack} The author is deeply grateful to Craig Guilbault for his help in the presentation of the material, as well as his useful comments and suggestions. Also, many thanks to Chris Hruska for his input in helping simplify the proof of Lemma 3.1 and Mladen Bestvina for his helpful suggestions on the exposition. 
\end{ack}

\section{Preliminaries}

We begin with a few preliminary definitions and then present some generalizations of Bestvina's original definition of a $\Z$-structure. 

We suppose that our spaces are locally compact, separable, and metrizable. First, recall that 
a separable metric space $X$ is an \emph{absolute retract} (or AR) if, whenever $X$ is embedded as a closed subset of another separable metric space $Y$, its image is a retract of $Y$ and $X$ is an \emph{absolute neighborhood retract} (or ANR) if some neighborhood of $X$ in $Y$ retracts onto $X$. If $X$ is a finite-dimensional AR, we call $X$ a \emph{Euclidean retract} (or ER). For more on these concepts, see for example \cite{Hu65}. 

 A closed subset, $A$, of an ANR, $X$, is a \emph{$\mathscr{Z}$-set} if there exists a homotopy $H:X\times [0,1]\to X$ such that $H_0=id_X$ and $H_t(X)\subset X-A$ for every $t>0$.
A \emph{$\mathscr{Z}$-compactification} of a space $Y$ is a compactification $\hat{Y}$ such that $\hat{Y}-Y$ is a $\mathscr{Z}$-set in $\hat{Y}$. 

\begin{defn} \cite{Be96} A \emph{$\mathscr{Z}$-structure on a group $G$} is a pair of spaces $(\hat{X},Z)$ satisfying the following four conditions:
 	\begin{enumerate}
		\item $\hat{X}$ is a compact ER,
		\item $\hat{X}$ is a $\Z$-compactification of $X=\hat{X}-Z$,
		\item $G$ acts properly, cocompactly, and freely on $X$, and 
		\item $\hat{X}$ satisfies a \emph{nullity condition} with respect to the action of $G$ on $X$. That is, for every compactum $C$ of $X$ and any open cover $\mathscr{U}$ of $\hat{X}$, all 		but finitely many $G$ translates of $C$ lie in an element of $\mathscr{U}$.
	\end{enumerate}
\end{defn}



We say that $Z$ is a \emph{boundary} (or $\Z$\emph{-boundary}) of $G$ if there is a $\Z$-structure $(\hat{X},Z)$ on $G$. This boundary is not unique; there can be multiple $\Z$-structures for a given group $G$. However, any two boundaries of $G$ will have the same shape \cite{Be96}.

In \cite{Dr06}, Dranishnikov generalized Bestvina's definition by omitting the requirement that $G$ act freely on $X$ and loosening the restriction that $\hat{X}$ be an ER to being an AR.  There is one immediate drawback in allowing infinite-dimensionality of $\hat{X}$: $Z$ could potentially be infinite-dimensional. We show in the next section that this is not the case; the covering dimension of $\Z$-boundaries (in the sense of Dranishnikov) is finite.  



Since our main result was motivated by attempting to generalize Swenson's theorem \cite[Theorem 12]{Swe99}, we need one final generalized definition of a $\Z$-structure that does not require properness of the action.


%

\begin{defn} Let $(X,d)$ be a metric space. A \emph{metric $\Z$-structure} on $X$, denoted $M\Z$-structure, is a pair of spaces $(\hat{X},Z)$ satisfying the following conditions:
	\begin{enumerate}
		\item $\hat{X}$ is a compact AR,
		\item $\hat{X}$ is a $Z$-compactification of $X=\hat{X}-Z$, 
		\item $X$ admits a cocompact action by isometries by some group $G$, and
		\item $\hat{X}$ satisfies a \emph{nullity condition} with respect to the action of $G$ on $X$:  for every $\epsilon>0$ and for each bounded subset $U$ of $X$ (bounded in the $d$ metric), there exists a compact subset $C$ of
		 $X$ such that any $G$-translate of $U$ that does not intersect $C$ has diameter less than $\epsilon$ (in the metric on the
		  compactification). 
	\end{enumerate}
\end{defn}

\section{Finite-Dimensionality Results}

Recall that a compact metric space has \emph{Lebesgue covering dimension at most $n$}, denoted dim$X\leq n$ if for every $\epsilon >0$, there exists an open cover $\mathscr{U}$ with mesh$(\mathscr{U})<\epsilon$ and order$(\mathscr{U})\leq n$. Here the \emph{order} of an open covering $\mathscr{U}$ being at most $n$ means that each $x\in X$ is in at most $n+1$ elements of $\mathscr{U}$ and the \emph{mesh} of a cover $\mathscr{U}$ is defined as mesh$(\mathscr{U})=\sup\{\text{diam}(U)|U\in\mathscr{U}\}$.

The goal of this section is to prove the following: 

\begin{thma} Let $(X,d)$ be a metric space which admits a $M\Z$-structure $(\hat{X},Z)$. Then \emph{dim}$Z<\infty$.
\end{thma}
Our proof relies on the following lemma. 



\begin{lemma} Suppose $G$ acts cocompactly by isometries on a proper metric space $X$. Then there exists a uniformly bounded open cover $\mathscr{U}$ of $X$ with finite order. \end{lemma}

\begin{proof} By cocompactness, we may choose $r>0$ so that $GB(x_0, r)=X$ for some $x_0\in X$. Let  $A$ be a maximal $r$-separated subset of the orbit of $x_0$ and $\mathscr{U}=\{B(x, 2r)|x\in A\}$. Clearly, $\mathscr{U}$ is a  uniformly bounded open cover.
We claim order$\mathscr{U}\leq n$, where $n$ is the maximal number of $r$-separated points in $\overline{B(x_0,4r)}$. Otherwise, there are points $x_1, x_2, ...x_{n+1}\in A$ with $\cap_{i=1}^{n+1}B(x_i, 2r)\neq\emptyset$. Thus, $r\leq d(x_i,x_j)<4r$ for $i\neq j$ and $i,j\in\{1,2,...,n+1\}$. Choosing an isometry $g\in G$ with $gx_1=x_0$, the points $gx_1,gx_2,...,gx_{n+1}$ are $r$-separated and contained in $B(x_0, 4r)$, a contradiction. 
\end{proof}


\begin{proof} [Proof of Theorem A] Let $H:\hat{X}\times[0,1]\to\hat{X}$ be a $Z$-set homotopy with $H_0=id_{\hat{X}}$ and $H_t(\hat{X})\cap Z=\emptyset$ for every $t>0$. Let $\epsilon>0$ and fix a metric $\hat{d}$ on $\hat{X}$.

Let $\mathscr{U}$ be a cover of $X$ as in the proof of Lemma 3.1 and $k<\infty$ be the order of $\mathscr{U}$. Using the nullity condition, we may choose a compactum $K\subseteq X$ such that diam$_{\hat{d}} V<\epsilon/2$ for every $V\in\mathscr{U}$ with $V\cap K=\emptyset$.

Choose $\delta_1\in(0,1]$ small enough such that $H_{\delta}(Z)$ is covered by open sets $V\in\mathscr{U}$ with diam$_{\hat{d}} V<\epsilon/2$ for all $\delta\leq\delta_1$.  Moreover, $H:\hat{X}\times[0,1]\to\hat{X}$ is uniformly continuous, so we may choose $\delta_2\in (0,1]$ so that for every $\delta\leq \delta_2$ and for each $z\in Z$, $\hat{d}(z, H_{\delta}(z))<\epsilon/4$. Set $t_{\epsilon}=$min$\{\delta_1,\delta_2\}$.

%


 Consider $\mathscr{V}_{\epsilon}=\{V\in\mathscr{U}|V\cap H_{t_{\epsilon}}(Z)\neq\emptyset \text{ and } V\cap K=\emptyset\}$. Notice $\mathscr{V}_{\epsilon}$ is an open cover of $H_{t_{\epsilon}}(Z)$ with mesh bounded by $\epsilon/2$ and order bounded by $k$. 
 
 Define $\mathscr{W}_{\epsilon}=\{H_{t_{\epsilon}}|_Z^{-1}(V)|V\in \mathscr{V}_{\epsilon}\}$. We show this is the desired cover. Clearly, $\mathscr{W}_{\epsilon}$ forms an open cover of $Z$ since $\mathscr{V}_{\epsilon}$ forms an open cover of $H_{t_{\epsilon}}(Z)$.  Moreover, mesh$_{\hat{d}}\mathscr{W}_{\epsilon}<\epsilon$ by the triangle inequality. 
Lastly, we know the order of the cover $\mathscr{V}_{\epsilon}$ of $H_{t_{\epsilon}}(Z)$ is at most $k$. Since $\mathscr{W}_{\epsilon}$ is the set of pre-images of $\mathscr{V}_{\epsilon}$ under the continuous map $H_{t_{\epsilon}}|_Z$, then $\mathscr{W}_{\epsilon}$ also has order at most $k$.

\end{proof}

\begin{rk} Theorem 12 in \cite{Swe99} now follows directly from Theorem A. 
\end{rk}
\begin{cor} If $G$ admits a $\Z$-structure $(\hat{X},Z)$ in the sense of \cite{Dr06}, then $\emph{dim}Z<\infty$. 
\end{cor}

\section{Consequences of Finite-Dimensionality of $\Z$-Boundaries}

We may now discuss how knowing finite covering dimension of the various $\Z$-boundaries can serve to unify the theories of group boundaries presented by Bestvina and Dranishnikov. First, any result about the cohomological dimension of the boundary in \cite{Dr06} may now be replaced with a statement concerning the Lebesgue covering dimension because in a space with finite Lebesgue covering dimension, covering dimension and cohomological dimension coincide (see for example \cite[Theorem 3.2(b)]{Wa81}). 

Secondly, there is no advantage in restricting ourselves to working with an ER. 

\begin{thmb} Suppose a group $G$ admits an AR $\Z$-structure. Then $G$ admits a $\Z$-structure.\end{thmb}

The proof of Theorem B relies on a more general version of Bestvina's boundary swapping theorem. Given that $G$ admits a $\Z$-structure, the original version of boundary swapping \cite{Be96} provides a method to take the boundary from the $\Z$-structure and place it on another finite-dimensional space admitting an action by $G$ to obtain a new $\Z$-structure on $G$. 
In the presence of finite-dimensionality of the boundary, the local contractibility condition for finite-dimensional ANRs is satisfied (\cite[Page 168]{Hu65}), and thus the proof of \cite[Lemma 1.4]{Be96} applies to prove the following version of boundary swapping. 

\begin{thm}[Boundary Swapping]
Let $G$ be a group acting properly, cocompactly, and freely on an ER $X_1$ and an AR $X_2$. Assume that $X_1$ and $X_2$ are $G$-homotopy equivalent and $\hat{X_2}=X_2\cup Z$ is an AR $\Z$-structure on $G$. Then $(\hat{X_1},Z)$ is a $\mathscr{Z}$-structure on $G$.  
\end{thm}

%



\begin{proof}[Proof of Theorem B] Let $(\hat{X},Z)$ be an AR $\Z$-structure for $G$. The map $p:X\to X/G$ is a covering projection, so $X/G$ is a compact ANR. Using a result from West \cite[Corollary 5.3]{We77}), $X/G$ is homotopy equivalent to a finite complex $Y$. Lifting the homotopies to the universal cover $\tilde{Y}$, an ER, we obtain a $G$-equivariant homotopy equivalence between $X$ and $\tilde{Y}$. Applying Theorem 4.1, $\tilde{Y}\cup Z$ is a $\Z$-structure for $G$. \end{proof}

\bibliography{Biblio}{}
\bibliographystyle{amsalpha}

\end{document}